\documentclass[11pt]{amsart}

\usepackage[dvipdfm]{hyperref}

\usepackage{mathptmx}
\usepackage{amsmath}
\usepackage{amssymb}
\usepackage{amsthm}
\usepackage{array}
\usepackage{enumerate}
\usepackage{graphicx}
\usepackage{lmodern}
\usepackage{mathrsfs}

\hypersetup{colorlinks,citecolor=blue,linkcolor=blue}

\newtheorem{theorem}{Theorem}[section]
\newtheorem{lemma}[theorem]{Lemma}
\newtheorem{prop}[theorem]{Proposition}
\newtheorem{cor}[theorem]{Corollary}

\theoremstyle{remark}
\newtheorem{rmk}{Remark}

\numberwithin{equation}{section}

\DeclareMathOperator{\SL}{SL}

\DeclareMathOperator{\PO}{PO}

\DeclareMathOperator{\PU}{PU}

\DeclareMathOperator{\PSL}{PSL}
\DeclareMathOperator{\PGL}{PGL}

\DeclareMathOperator{\vol}{vol}
\newcommand{\tors}{\mathrm{tors}}
\newcommand{\free}{\mathrm{free}}
\newcommand{\im}{\mathrm{im}}

\newcommand{\Z}{\mathbb Z}
\newcommand{\Q}{\mathbb Q}
\newcommand{\R}{\mathbb R}
\newcommand{\C}{\mathbb C}
\newcommand{\A}{\mathbb A}

\renewcommand{\O}{\mathcal O}
\newcommand{\D}{D}
\newcommand{\q}{\mathfrak q}
\newcommand{\p}{\mathfrak p}

\newcommand{\Af}{\A_\mathrm{f}}

\newcommand{\V}{V}
\newcommand{\Vf}{\V_{\mathrm{f}}}
\newcommand{\Sf}{S_{\mathrm{f}}}
\newcommand{\Vi}{\V_\infty}

\newcommand{\GG}{\mathbf G}

\newcommand{\bs}{\backslash}

\begin{document}
\title[Torsion homology and $K_2$ of imaginary fields]{Torsion homology of arithmetic lattices and $K_2$ of imaginary fields}

\begin{abstract}
  Let $X = G/K$ be a symmetric space of noncompact type.
  A result of Gelander provides exponential upper bounds in terms of the
  volume for the torsion homology of the noncompact locally symmetric spaces
  $\Gamma\bs X$. We show that under suitable assumptions on $X$ this
  result can be extended to the case of nonuniform lattices $\Gamma
  \subset G$ that may contain torsion. Using recent work of Calegari and
  Venkatesh we deduce from this upper bounds (in terms of the
  discriminant) for $K_2$ of the ring of
  integers of totally imaginary number fields $F$. More generally, we obtain
  such bounds for rings of $S$-integers in $F$. 
\end{abstract}

\author{Vincent Emery}
\thanks{Supported by Swiss National Science Foundation, Projects number
  {\tt PA00P2-139672} and {\tt PZ00P2-148100}}

\address{
Institut Mathgeom\\
EPF Lausanne\\
B\^atiment MA, Station 8\\
CH-1015 Lausanne\\
Switzerland
}
\email{vincent.emery@gmail.com}

\date{\today}


\maketitle

\section{Introduction}
\label{sec:intro}

Let $\GG$ be a connected semisimple real algebraic group such that
$\GG(\R)$ has trivial center and no compact factor. We denote by $X$ the
associated symmetric space, i.e., the homogeneous space $X = \GG(\R)/K$
where $K$ is a maximal compact subgroup $K \subset \GG(\R)$. 
Many properties of torsion-free lattices $\Gamma \subset \GG(\R)$ can 
be studied with the help of geometry,
by analyzing the corresponding locally symmetric spaces $\Gamma \bs X$. 
For example, a theorem of Gromov shows that the Betti numbers of
$\Gamma$ are bounded linearly in the covolume, i.e., there exists a
constant $C_\GG$ such that 
\begin{align}
 \dim H_j(\Gamma,\Q) & \le C_\GG \vol(\Gamma\bs X).
  \label{eq:Betti}
\end{align}
Recently, Samet has extended this result to the case of lattices that may
contain torsion \cite{Samet}. Note that these results are valid in a 
more general context than symmetric spaces. 

Lately, much interest in the torsion part in the integral homology
of lattices -- especially arithmetic lattices -- has arisen,
due to connection with number theory (cf. \cite{BergVenka,CalVen}).
We will denote by $H_j(\Gamma)$ the homology with integral coefficients,
and by $H_j(\Gamma)_\tors$ its torsion part. In \cite{Gel04} Gelander
proved that each noncompact arithmetic manifold has the homotopy type of
a simplicial complex whose size is linearly bounded by the
volume. This shows (cf. Lemma \ref{lemma:Gabber}) that for
nonuniform torsion-free arithmetic lattices,
$\log|H_j(\Gamma)_\tors|$ is linearly bounded by $\vol(\Gamma\bs X)$.
Note that $\log|\cdot|$ for finite sets is analogue to the dimension
and thus this result provides a version of \eqref{eq:Betti} for the torsion.
The corresponding result for compact manifolds would follow quite easily from
Lehmer's conjecture on Mahler measure of integral polynomials (cf.
\cite[\S 10]{Gel04}). 

In this article we show that Gelander's result can also
be used -- at least for suitable $\GG$ -- to bound the torsion homology
of nonuniform arithmetic lattices $\Gamma \subset \GG(\R)$, without the
restriction that $\Gamma$ is torsion-free. Namely, in \S
\ref{sec:bound-manifolds} we prove the following theorem. Recall that a
lattice $\Gamma \subset \GG(\R)$ is called irreducible if it is dense
in each nontrivial direct factor of $\GG(\R)^\circ$. If $\GG(\R)$ is
simple then all its lattices are irreducible.

\begin{theorem}
  Let $\GG$ be as above such that for all irreducible lattices 
  $\Gamma \subset \GG(\R)$ we have $H_q(\Gamma,\Q) = 0$ for $q=1,\dots,j$. 
  Then, there exists a constant $C_\GG > 0$ such that for each irreducible
  nonuniform arithmetic lattice $\Gamma
  \subset \GG(\R)$ the following bound on torsion homology holds: 
  \begin{align*}
    \log \left| H_j(\Gamma) \right| &\le 
    C_\GG\; \vol(\Gamma\bs X). 
  \end{align*}
  \label{thm:torsion}
\end{theorem}

Superrigity of lattices (which holds for all $\GG(\R)$
except $\PO(n,1)$ and $\PU(n,1)$) implies at once arithmeticity and the
vanishing of the first Betti number. Thus, for the first homology our
theorem reads:
\begin{cor}
 If $\GG(\R)$ is not locally isomorphic to $\PO(n,1)$ or $\PU(n,1)$
 then there exists a constant $C_\GG > 0$ such that for each
 irreducible nonuniform lattice $\Gamma \subset \GG(\R)$ we have
  \begin{align*}
    \log \left| H_1(\Gamma) \right| &\le C_\GG\; \vol(\Gamma\bs X). 
  \end{align*}
\end{cor}
It is possible that the condition $H_q(\Gamma,\Q) = 0$ in Theorem
\ref{thm:torsion} is not necessary. However, it is not clear from our
proof how to remove it (see Remark \ref{rmk:Jones} at the end of
\S \ref{sec:bound-manifolds}).

One motivation for Theorem \ref{thm:torsion} is of arithmetic nature: 
it concerns the $K$-theory of number fields. Let $F$ be a
number field with ring of integers $\O_F$.
We denote by $D_F$ the discriminant of $F$ and by $w_F$ the number of roots of
unity in $F$. The group $K_2(\O_F)$ is known to be finite and to injects as a 
subgroup -- the tame kernel -- of $K_2(F)$ (see \cite[\S 5.2]{Weib05}). 
By a theorem of Suslin (cf. \cite[Theorem 4.12]{CalVen}),
$K_2(F)$ corresponds to the second homology of $\PGL_2(F)$.  
Using this as a starting point,
in \cite[\S 4.5]{CalVen} Calegari and Venkatesh have been able to relate
in turn $K_2(\O_F)$ to the second homology of $\PGL_2(\O_F)$.
From their results and Theorem~\ref{thm:torsion},
we will show that the following holds.

\begin{theorem}
  Let $d \ge 2$ be an even integer.
  There exists a constant $C(d) > 0$ such that
  for each totally imaginary field $F$ of degree $d$  we have:
  \begin{align}
    \log|K_2(\O_F) \otimes R| &\le C(d) |\D_F|^{2} (\log|D_F|)^{d-1},
    \label{eq:bound-K2}
  \end{align}
  where $R = \Z[\frac{1}{6w_F}]$.
  \label{thm:K2}
\end{theorem}

For  $d \ge 6$ the theorem follows immediately from the
mentionned results by noticing that the Betti number of $\PGL_2(\O_F)$
in degree $j=1, 2$ vanish in this case. This fact is not only needed to
satisfy the conditions of Theorem \ref{thm:torsion} but also to apply
the work of Calegari and Venkatesh, which in particular requires the
congruence subgroup property (it fails for $d = 2$). This also explains 
the appearence in the statement of the value $w_F$, which corresponds to
the order of the congruence kernel.

If $S$ is a finite set of places of $F$, for the set $\O_F(S)$ of $S$-integers 
of $F$ we have $K_2(\O_F) \subset K_2(\O_F(S))$ (cf. \cite[Theor\`eme
1]{Soule84}).
To deal with the cases $d = 2,4$ we will then consider rings of $S$-integers.
As far as we know, there is no known counterpart to Gelander's result
\cite{Gel04} for $S$-arithmetic groups. However, for groups of type
$A_1$ (e.g., $\SL_2(\O_F(S))$) we can use their action on Bruhat-Tits trees to write
them as almagamated products and obtain this way upper bounds for their
homology torsion (see Section \ref{sec:S-arithm}). As a corollary we obtain in
Theorem \ref{thm:K2-S-arithm} a generalization  of Theorem \ref{thm:K2}
for rings of $S$-integers. There again, the work of Calegari and
Venkatesh is essential in our proof.



General upper bounds for the $K$-theory of rings of integers have
been obtained by Soul\'e in \cite{Soule03}, without restriction on the
signature of $F$. His method was later improved by Bayer and Houriet 
(their results are contained in the second part of Houriet's thesis, see
\cite[Theorem 4.3]{Houriet}).
For $K_2$ of totally imaginary fields, Theorem \ref{thm:K2} improves
considerably these known bounds -- at least asymptotically (since we do
not provide an explicit value for $C(d)$).  Still, our bounds might be
very far from being sharp. Explicit computations for $K_2(\O_F)$ are very
difficult but some were obtained by Belabas and Gangl in
\cite{BelGangl}, mostly for $F$ imaginary quadratic. Their result tends
to show that $K_2(\O_F)$ has a quite slow growing rate.

Recall that for the order of $K_0$, that is, the class number
$h_F$, we have the following:
\begin{align}
  h_F &\le C(d) |D_F|^{1/ 2} (\log|D_F|)^{d-1},
  \label{eq:class-no}
\end{align}
for some constant $C(d)$. For higher degrees $m$, Soul\'e has
conjectured (see \cite[\S 5]{Soule03}) that the torsion part of
$K_m(\O_F)$ can be bounded as follows:
\begin{align}
  \log|K_m(\O_F)_\tors| &\le C(m, d) \log|D_F|,
  \label{eq:conjecture-Soule}
\end{align}
for some constant $C(m, d)$ and $F$ of degree $d$ (and arbitrary
signature). Such a bound is out of reach with the method presented in
this article, and some new ideas will be needed to further improve the
situation.




\subsection*{Acknowledgements} I would like to thank Akshay Venkatesh
for his help and encouragement, and Mike Lipnowski, Jean
Raimbault and Misha Belolipetsky for helpful comments.

\section{Bounds for torsion homology}
\label{sec:bound-manifolds}


Let $\GG$ and $X$ be as in the introduction.  
For a torsion-free discrete subgroup  $\Gamma \subset \GG(\R)$ the
quotient $M = \Gamma \bs X$ is a manifold locally isometric to $X$. We
call such an $M$ a $X$-manifold. $M$ is called arithmetic if $\Gamma$ is
an arithmetic subgroup of $\GG(\R)$. In this case $M$ has finite volume.
The following result, proved in \cite{Gel04},
gives a strong quantitative relation between the
geometry and topology of noncompact arithmetic manifolds.

\begin{theorem}[Gelander]
  There exists a constant $\beta = \beta(X)$ such that any noncompact
  arithmetic $X$-manifold $M$ is homotopically equivalent to a
  simplicial complex $\mathcal K$ whose numbers of $q$-cells is at most
  $\beta \vol(M)$ for each $q \le \dim(X)$. \label{thm:Gelander}
\end{theorem}

In particular, $\beta \vol(M)$ is an upper bound for the Betti numbers,
a result already known from the work of Gromov in a larger context.
But Theorem \ref{thm:Gelander} also allows us to study the torsion part
in the homology. For this we will use the following result, whose 
proof can be found in \cite[\S 2.1]{Soule99}.
For an abelian group $A$, we denote by $A_\tors$ its subgroup of torsion
elements.
\begin{lemma}[Gabber]
  \label{lemma:Gabber}
  Let  $A = \Z^a$ with the standard basis $(e_i)_{i=1,\dots,a}$
  and $B = \Z^b$, so that $B \otimes \R$ is equipped with the standard
  Euclidean norm $\|\cdot\|$.  Let $\phi: A \to B$ be a $\Z$-linear map
  such that  $\| \phi(e_i) \| \le \alpha$ for each $i = 1, \dots, a$. 
  If we denote by $Q$ the cokernel of $\phi$, then 
  \begin{align*}
    |Q_\tors| &\le  \alpha^{\min\{a,b\}}\;.
  \end{align*}
\end{lemma}
This lemma applies in particular to simplicial complexes, where
the boundary map on a $q$-simplex is a sum of $(q+1)$ basis elements.
From Theorem \ref{thm:Gelander} we obtain a bound 
\begin{align}
  H_q(M)_{\tors} &\le \alpha^{\vol(M)},
  \label{eq:torsion}
\end{align}
for some $\alpha$ that depends only on $X$. We want to extend this
result to lattices $\Gamma \subset \GG(\R)$ that may contain torsion.
Note that there exists a bound $\gamma = \gamma(\GG)$ such that each
nonuniform arithmetic $\Gamma \subset \GG(\R)$ contains a torsion-free
normal subgroup $\Gamma_0$ of index $[\Gamma:\Gamma_0]
\le \gamma$. This is a well-known fact  that  follows from the existence 
of a $\Q$-structure on $\GG$ such that $\Gamma \subset \GG(\Q)$
(cf.\cite[Lemmas 5.2 and 13.1]{Gel04}). This proves:
\begin{prop}
  \label{prop:bound-homol-OF}
 There exist $\alpha, \gamma >0$ depending only on $\GG$ such that
 for each nonuniform arithmetic  lattice $\Gamma \subset \GG(\R)$  
 there is a $X$-manifold $M$ that is a normal cover of
 $\Gamma \bs X$, with Galois group $G$, such that: 
 \begin{enumerate}
   \item the order of $G$ is bounded by $\gamma$;
   \item the order of $H_q(M)_\tors$ is bounded by
     $\alpha^{\vol(\Gamma\bs X)}$ for each $q$.
 \end{enumerate}
\end{prop}

For a given nonuniform arithmetic lattice $\Gamma \subset \GG(\R)$,
let the manifold $M$ and the group $G$ be as in
Proposition \ref{prop:bound-homol-OF}. Thus we have the exact sequence
\begin{align}
 1 \to \pi_1(M) \to \Gamma \to G \to 1, 
  \label{eq:ex-sq-covering}
\end{align}
and the homology of $\Gamma$ can be studied with help of the Lyndon-Hochschild-Serre
spectral sequence (see \cite[\S VII.6]{Brown82}):
\begin{align}
  E^2_{pq} = H_p(G, H_q(M))\; & \Rightarrow \; H_{p+q}(\Gamma).
  \label{eq:spectral-seq}
\end{align}
In particular, the order of $H_j(\Gamma)_\tors$ is equal to
$\prod_{p+q = j} |(E^\infty_{pq})_\tors|$. Since the homology of the
finite group $G$ is torsion outside the degree $0$, the second page
$E^2_{pq}$ has no infinite factor outside the vertical line $p = 0$.
It follows that the successive (diagonal) differentials on the spectral 
sequence \eqref{eq:spectral-seq}  do not add any torsion, so that the 
torsion in $E^\infty_{pq}$ is bounded by the torsion in $E^2_{pq}$. Thus,
\begin{align}
  |H_j(\Gamma)_\tors| & \le \prod_{p+q = j} \left|H_p(G, H_q(M))_\tors\right|.
  \label{eq:bound-by-eq-seq}
\end{align}

We can now conclude the proof of Theorem \ref{thm:torsion}.
The factor $|H_j(G,\Z)_\tors|$ in \eqref{eq:bound-by-eq-seq} depends only 
on $j$ and $G$, and since there exist only a finite number of groups of order
less than $\gamma$, we have a uniform bound for it. For $q = 1,\dots,
j$, our hypothesis on the Betti numbers of  lattices in $\GG(\R)$ implies 
that the abelian group $A = H_q(M)$ is finite.
It follows that $H_p(G,A)_\tors$ is equal to
$H_p(G,A)$. This homology group can be computed by  
by the standard bar complex (see \cite[\S III.1]{Brown82}),
which in degree $p$ corresponds to finite sums  of
symbols $a \otimes [g_1|\cdots|g_{p}]$, with $a \in A$ and $g_i \in G$.
In particular, this shows that the order of $H_p(G,A)$ is at most
$|A|^{|G|^{p}}$, which by Proposition \ref{prop:bound-homol-OF}
is bounded by $\alpha^{\gamma^{p} \vol(\Gamma \bs X) }$.
This concludes the proof of Theorem .

\begin{rmk}
  If one forgets about the small primes, the proof can be
  simplified. Let $N$ be the product of all primes bounded by
  $\gamma$ (it depends on $\GG$ only). 
  The transfer map for homology (cf. \cite[\S III.9--10]{Brown82})
  shows that $H_j(\Gamma,\Z[1/N])$ is given by the module of
  co-invariants $H_j(M, \Z[1/N])_G$, which is isomorphic to the
  submodule of invariants $H_j(M, \Z[1/N])^G$ (since $|G|$ is invertible
  in $\Z[1/N]$). 
  So we may use $|H_j(M,\Z[1/N])|$ as an upper bound, and the latter  
  is bounded by Proposition \ref{prop:bound-homol-OF}.
  In particular this argument does not need the vanishing of the Betti
  numbers.
 \label{rmk:transfer}
\end{rmk}

\begin{rmk}
  \label{rmk:Jones}
 To avoid the hypothesis $H_q(M,\Q) = 0$  in our proof we
 would need to obtain a bound for the torsion in $H_p(G,A_\free)$, where
 $A_\free$ is the free part of the $G$-module $A = H_q(M)$.
 The boundary map on a basis element in the bar complex for
 $H_p(G,A_\free)$ is given by 
 \begin{align}
   \partial(a \otimes [g_1|\cdots|g_p]) &=
   a g_1 \otimes [g_2|\cdots|g_p] - a \otimes [g_1|g_3|\cdots|g_p] +
   \dots \label{eq:boundary-bar}
 \end{align}
 The problem that prevents us to apply Lemma \ref{lemma:Gabber} in this
 context to
 obtain a good bound lies in the first term: even though the group $G$
 is finite of order uniformly bounded by $\gamma$, we cannot bound
 the size of $a g_1$ uniformly. The reason is that in general
 (indecomposable) integral representations of a finite group
 can be arbitrarily large (see \cite[Theorem (81.18)]{CurtRein62}).
\end{rmk}


\section{Bounds for  $K_2$  of imaginary number fields. The case $d \ge 6$.}
\label{sec:K2} 

In this section we prove Theorem \ref{thm:K2} for fields $F$ of degree
$d \ge 6$. The cases $d = 2$ and $d=4$ are dealt with as a special case
of Theorem \ref{thm:K2-S-arithm}, which will require more preparation.  

Let $F$ be a number field with class number $h_F$.  We denote by $\Vf$
(resp. $\Vi$) the set of finite (resp. archimedean) places of $F$ and by 
$\Af$ its ring of finite ad\`eles.
Let $G$ be the algebraic group
$\PGL_2$ defined over $F$, and let $\GG(\R) = \PGL_2(F \otimes_\Q \R)$
be the associated real Lie group (which depends only on the signature of
$F$) and $X$ its associated symmetric space (which is a product of
hyperbolic spaces of dimensions $2$ and $3$).
In \cite{CalVen} Calegari and Venkatesh consider the following double
cosets space:
\begin{align}
  Y_0(1) &= G(F)\bs\left(X \times G(\Af)/K  \right),
  \label{eq:Y_0}
\end{align}
where $X$ is identified as the quotient of $\GG(\R)$ by a maximal
compact subgroup, and $K \subset G(\Af)$ is the compact open subgroup 
$K = \prod_{v \in \Vf} \PGL_2(\O_v)$. Then, $Y_0(1)$ consists of a disjoint
union of copies of the quotient $\PGL_2(\O_F)\bs X$, and the number of
connected components is indexed by the finite set
\begin{align}
  G(F)\bs G(\Af) / K.
  \label{eq:indexed-double-coset}
\end{align}
In our case $G = \PGL_2$, the size of this set corresponds to the order
of the class group $\mathrm{Cl}(F)$ modulo its square elements. In particular, we
can use the bound $|\pi_0(Y_0(1))| \le h_F$. Moreover, if $h_F$ is odd
then $Y_0(1)$ is connected. 

We follow the convention used in \cite{CalVen} that the homology is
understood in the orbifold sense. For example, if $h_F$ is odd then
$H_\bullet(Y_0(1))$ corresponds exactly to the group homology of $\PGL_2(\O_F)$.  
For $F$ with at least two archimedean places,
Calegari and Venkatesh prove (cf. \cite[Theorem 4.5.1]{CalVen})
the existence of a surjective map 
\begin{align}
 H_2(Y_0(1),R) &\to K_2(\O_F) \otimes R, 
  \label{eq:surj-map}
\end{align}
where $R = \Z[\frac{1}{6w_F}]$.

Suppose that $F$ is totally imaginary of degree $d \ge 6$. Then, $\GG(\R)$
has real rank at least $3$, and from the discussion in
\cite[\S XIV.2.2]{BorWal} we know that for each irreducible arithmetic 
subgroup
$\Gamma \subset \GG(\R)$ there is in degree $j=1,2$ a surjective homomorphism 
\begin{align}
  H^j(X_u,\C) &\to H^j(\Gamma,\C),
  \label{eq:compact-dual}
\end{align}
where $X_u$ denotes the compact dual of $X$.
But since $F$ is totally imaginary, $X$ is a product of hyperbolic $3$-spaces 
and thus $X_u$ is a product of $3$-spheres. 
It follows from the K\"unneth formula
that $H^j(\Gamma,\C) = 0$ for $j=1,2$. This shows that Theorem
\ref{thm:torsion} applies to the groups $\PGL_2(\O_F)$.
This also shows that the surjective
map \eqref{eq:surj-map} provides in this case the bound 
\begin{align}
 |K_2(\O_F) \otimes R| &\le |H_2(Y_0(1),R)| \nonumber\\
 &\le |H_2(\PGL_2(\O_F))|^{h_{F}}.
  \label{eq:bound-K2-homol}
\end{align}

The covolume of $\PGL_2(\O_F)$ is well known and has been computed for
instance in \cite{Bor81}. We can  bound it above by $|D_F|^{3 / 2}$.
The claim in Theorem \ref{thm:K2} then follows from Theorem \ref{thm:torsion}, 
and the bounds \eqref{eq:bound-K2-homol} and \eqref{eq:class-no}. 

\begin{rmk}
 Note that the proof shows that we have the better bound   
 \begin{align}
   \log |K_2(\O_F) \otimes R| &\le  C(d) |D_F|^{3 / 2} 
   \label{eq:K2-odd-class-no}
 \end{align}
 for the totally imaginary fields $F$ of degree $d\ge6$ and odd class number.  
\end{rmk}

\section{Two lemmas about torsion}

We state here two lemmas that will be needed in Section 
\ref{sec:S-arithm}. If $A$ is a module, we denotes by $A_\free$ its
free part: $A_\free = A/A_\tors$. 

\begin{lemma}
  \label{torsion-exact-sq}
  Let $A \stackrel{\phi}{\to} B \to C \to D$ be an exact sequence of finitely generated
  $\Z$-module, and denote by $Q$ the cokernel of the map $\phi_\free:
  A_\free \to B_\free$ induced by $\phi$. Then,
  \begin{align*}
    |C_\tors| &\le  |Q_\tors|\cdot |B_\tors|\cdot |D_\tors|.
  \end{align*}
\end{lemma}

\begin{proof}
  Let us use the notation $\beta: B \to C$ and $\gamma: C \to D$.
  Since $C/\im(\beta) \cong \im(\gamma) \subset D$, by restricting to
  the torsion elements we find 
  \begin{align*}
    |C_\tors| &\le |\im(\beta)_\tors| \cdot |D_\tors|.
  \end{align*}
  We can decompose each element of $c \in \im(\beta)_\tors$ into 
  $c = \beta(b_\free) + \beta(b_\tors)$, where $b_\tors$ is a torsion
  element and $b_\free$ is not. Then, $|B_\tors|$ is obviously an 
  upper bound for the number of elements of the form $\beta(b_\tors)$.
  By exactness of the sequence, the elements in the coset  $b_\free +
  \phi(A_\free)$ all have same image $\beta(b_\free) \in C$, so that there are
  at most $|Q_\tors|$  elements $\beta(b_\free) \in C_\tors$.
\end{proof}

\begin{lemma}
  \label{lemma:torsion-from-inclusion}
  Let $\Gamma$ be a group and $\Gamma_0 \subset \Gamma$ a normal
  subgroup of finite index. We consider the map $\phi:
  H_j(\Gamma_0)_\free \to H_j(\Gamma)_\free$  induced by the
  inclusion on the free part of the homology. Then, the order 
  of the cokernel of $\phi$ is bounded above by
  $[\Gamma:\Gamma_0]^{b_j(\Gamma)}$.
\end{lemma}

\begin{proof}
  The result follows from the existence of 
  the \emph{transfer} map $\tau : H_j(\Gamma) \to H_j(\Gamma_0)$,
  for which $(\phi \circ \tau)(z) = [\Gamma:\Gamma_0] \cdot z$ for any
  $z \in H_j(\Gamma)$ (cf. \cite[\S III.9]{Brown82}): if $M =
  H_j(\Gamma)_\free$ then we have 
  \begin{align*}
    (\phi\circ\tau)(M) \;\subset\; \phi(H_j(\Gamma_0)_\free)
    \;\subset\; M,
  \end{align*}
  with $[M : (\phi\circ\tau)(M)] = [\Gamma:\Gamma_0]^{b_j(\Gamma)}$.
\end{proof}

\section{Torsion homology of $S$-arithmetic groups of type $A_1$} 
\label{sec:S-arithm}

To finish the proof of Theorem \ref{thm:K2} we will need to
consider $S$-arithmetic groups. Let $F$ be a number field.
For a finite set $S$ of places of $F$
containing all the archimedean, we denote by $\O_F(S) = \O_F[S^{-1}]$
the ring of $S$-integers in $F$. We write $\Sf$ for the subset $\Sf
\subset S$ of finite places.
We define the \emph{norm} of $S$ in $F$ by 
\begin{align}
  N_{F/\Q}(S)  &= \prod_{\p \in \Sf} N_{F/\Q}(\p),
  \label{norm-S}
\end{align}
where $N_{F/\Q}(\p)$ is the norm of the ideal $\p$. 

For a given $S$ we consider the $S$-arithmetic group $\SL_2(\O_F(S))$
and its subgroups of finite index. They act on the space $X_S = X
\times X_{\Sf}$, where $X = X_\infty$ is the symmetric space associated with
$\SL_2(F \otimes_\Q \R)$ and 
\begin{align}
  X_{\Sf} &= \prod_{\p \in \Sf} X_\p,
  \label{eq:X_Sf}
\end{align}
where $X_\p$ is the Bruhat-Tits tree associated with the $\p$-adic
group $\SL_2(F_\p)$.

\begin{prop}
 Let $d,N>0$ be two integers. There exists a constant $C(d,N)$ such that
 for any number field $F$ of degree $d$ and any irreducible torsion-free
 $S$-arithmetic subgroup $\Gamma \subset \SL_2(\O_F(S))$ with $N_{F/\Q}(S) \le N$
 we have that for any $j > 0$ the Betti number $b_j(\Gamma)$ and 
 $\log |H_j(\Gamma)_\tors|$ are bounded above by
 \begin{align*}
   C(d,N) \cdot [\SL_2(\O_F(S)):\Gamma]  \cdot |D_F|^{3/2}.
 \end{align*}
   \label{prop:bound-S-arithm}
\end{prop}

\begin{proof}
  For a fixed degree $d$ we proceed by induction on $N$.
  If $N = 1$, we have that $\Sf$ is empty and so $\Gamma$ is a genuine arithmetic
  group in $\SL_2(F \otimes_\Q \R)$. Moreover, since $\Gamma$ is
  torsion-free, it is isomorphic to its image in $\PGL_2(F \otimes_\Q
  \R)$. The result in this case follows then from the bound
  \eqref{eq:torsion} explained in Section \ref{sec:bound-manifolds},
  together with the volume formula for $\SL_2(\O_F)$ (see \cite{Bor81}).

  Suppose that the result holds for any set $S'$ of places of norm
  $N_{F/\Q}(S')$ less than $N$,
  and let $\Gamma$ be an $S$-arithmetic group as in the statement, and with
  $N_{F/\Q}(S) = N > 1$. The cohomological dimension of $\Gamma$ is then
  bounded (see \cite[Proposition 21]{Serr71}) and thus it suffices to
  prove the induction step for a fixed $j$ (the constant $C(d,N)$ can be
  constructed as the maximum of the constants $C(d,N,j)$). 
  For some $\q \in \Sf$, let $S = S' \setminus
  \left\{ \q \right\}$. Since $\Gamma$ is $S$-arithmetic, it has finite
  index in $\SL_2(\O_F(S))$, and the latter being dense in $\SL_2(F_\q)$ by
  the strong approximation property (cf. \cite[Prop. 7.2 (1)]{PlaRap94}),
  we conclude that $\Gamma$ is dense in $\SL_2(F_\q)$.
  Then, by letting $\Gamma$ act on the tree $X_\q$ (see \cite[\S II.1.4]{Serre80}),
  it  can be written as an amalgamated product
  \begin{align}
    \Gamma &= \Gamma' *_{\Gamma'_0(\q)} \Gamma',
    \label{eq:amalg-prod}
  \end{align}
  where $\Gamma' = \Gamma \cap \SL_2(\O_F(S'))$ and $\Gamma'_0(\q)$ is
  its congruence subgroup of level $\q$, that is,
  \begin{align}
    \Gamma'_0(\q) &= \left\{  \left( \begin{array}{cc} a & b \\ c &
      d\end{array} \right) \Big|\; 
    c \equiv 0 \mod \q  \right\}.
    \label{eq:congr-subgp-q}
  \end{align}
  In particular, $\Gamma'$ and
  $\Gamma'_0(\q)$ are $S'$-arithmetic, with $N_{F/\Q}(S') < N$.
  Since the norm of $\q$ is bounded
  by $N$, the index $[\Gamma':\Gamma'_0(\q)]$ is uniformly bounded, say
  by $\gamma(d,N)$. Then, we have
  \begin{align}
    \nonumber
    [\SL_2(\O_F(S')):\Gamma'_0(\q)] &\le \gamma(d,N) \cdot [\SL_2(\O_F(S') :
      \Gamma']\\
      &\le \gamma(d,N) \cdot [\SL_2(\O_F(S)) : \Gamma].
    \label{eq:bounds-indices}
  \end{align}

  From \eqref{eq:amalg-prod} we obtain the following exact sequence
  in homology (see \cite[\S II.7]{Brown82}):
  \begin{align}
    H_j(\Gamma_0'(\q)) \to H_j(\Gamma') \oplus H_j(\Gamma') \to
    H_j(\Gamma) \to H_{j-1}(\Gamma_0'(\q)),
    \label{exact-seq-amalg-prod}
  \end{align}
  where the first map comes from the diagonal inclusion $\Gamma'_0(\q)
  \to \Gamma' \times \Gamma'$. By the recurrence hypothesis together
  with \eqref{eq:bounds-indices}, the rank  of $H_{j-1}(\Gamma'_0(\q))$ 
  is bounded by 
  \begin{align}
    \label{bound-S-prime}
    C(d,N-1) \gamma(d,N)  \cdot [\SL_2(\O_F(S)) : \Gamma] \cdot |D_F|^{3/2},
  \end{align}
  and so is also the rank $b_j(\Gamma')$ of $H_j(\Gamma')$. It follows
  form the exact sequence \eqref{exact-seq-amalg-prod} that
  $b_j(\Gamma)$ can be bounded as wanted.

  It remains to bound the torsion homology. Applying Lemma
  \ref{lemma:torsion-from-inclusion} we see that the torsion
  in $H_j(\Gamma')$ that comes from the free part of
  $H_j(\Gamma_0'(\q))$ is bounded by
  $[\Gamma:\Gamma_0]^{b_j(\Gamma')}$. Since by induction
  the expression \eqref{bound-S-prime} also serves as a bound for
  $\log|H_j(\Gamma')_\tors|$ and $\log|H_{j-1}(\Gamma_0'(\q))_\tors|$,
  by applying Lemma \ref{torsion-exact-sq} on \eqref{exact-seq-amalg-prod} 
  we obtain the needed bound for $\log|H_j(\Gamma)_\tors|$. This shows that
  both $b_j(\Gamma)$ and $\log |H_j(\Gamma)_\tors|$ can be bounded by
  \begin{align*}
    C(d,N) \cdot [\SL_2(\O_F):\Gamma] \cdot |D_F|^{3/2}
  \end{align*}
  for some constant $C(d,N)$
  constructed from $C(d,N-1)$ and $\gamma(d,N)$.
\end{proof}

As a corollary of Proposition \ref{prop:bound-S-arithm}, we can bound the
low degree torsion homology of the $S$-arithmetic groups $\PSL_2(\O_F(S))$ 
for totally imaginary fields.

\begin{theorem}
  Let $S$ be a set of places of the totally imaginary number field $F$ of degree $d$,
  containing all the archimedean. Suppose that $N_{F/\Q}(S) \le N$, and
  $|\Sf| \ge  3$. Then, for $j = 1, 2$ we have
  \begin{align*}
    \log|H_j(\PSL_2(\O_F(S)))| \le C(d,N) \cdot |D_F|^{3/2} 
  \end{align*}
  for some constant $C(d,N)$.
  \label{thm:torsion-SL-S}
\end{theorem}

\begin{proof}
  Let $\Gamma_S = \SL_2(\O_F(S))$. 
  We can choose a finite prime $\q$ outside of $S$, large enough so that
  the principal congruence subgroup $\Gamma_S(\q) \subset \Gamma_S$ is
  torsion-free. Moreover, $\q$ can be chosen of norm bounded above by 
  by some constant depending only on $d$ and $N$. In particular, there
  is some upper bound $\gamma(d, N)$ for the index $[\Gamma_S:\Gamma_S(\q)]$.
  It is known (cf. for instance \cite[proof of Theorem 2.1]{BorYang94})
  that for $q < |\Sf|$ the cohomology groups
  $H^q(\Gamma_S(\q), \R)$ correspond to the continuous cohomology of
  $\PGL_2(F \otimes_\Q \R)$, which equals the cohomology
  $H^q(X_u, \R)$ of the compact dual $X_u$ of $X_\infty$. Since $F$ is
  totally imaginary, we find (cf. Section \ref{sec:K2}) that the Betti
  numbers $b_q(\Gamma_S(\q))$ vanishes for $q = 1$ and $2$.

  Since $\Gamma_S(\q)$ is torsion-free, it is isomorphic to its image in
  $\PSL_2(\O_F(S))$, which we will denote by the same symbol. Then 
  $\Gamma_S(\q)$ is viewed as a normal subgroup of $\PSL_2(\O_F(S))$,
  whose index is bounded by $\gamma(d,N)$. Let $G$ be the quotient of
  $\PSL_2(\O_F(S))$ by $\Gamma_S(\q)$.
  Then, as was done in Section \ref{sec:bound-manifolds} for the proof of
  Theorem \ref{thm:torsion}, we can bound the torsion of
  $H_j(\PSL_2(\O_F(S)))$ in terms of $|H_q(\Gamma_S(\q))|$, using the spectral sequence
  \begin{align*}
    H_p(G, H_q(\Gamma_S(\q))) &\Rightarrow
    H_{p+q}(\PSL_2(\O_F(S))). 
  \end{align*}
  The conclusion then follows from the bound for the torsion homology of
  $\Gamma_S(\q)$ provided by Proposition \ref{prop:bound-S-arithm}.
\end{proof}


  

\section{$K_2$ of rings of $S$-integers in imaginary fields}

Let us write again $G$ for the algebraic group $\PGL_2$ over the number field $F$. 
We denote by $\A^{S}$ the  ad\`ele ring of $F$ omitting
the places $v \in S$ (where as usual we suppose $S$ finite with  $\Vi \subset S$).
For the open compact subgroup $K^S =
\prod_{v \not \in S} \PGL_2(\O_v) \subset G(\A^{S})$, following
\cite[\S 4.4]{CalVen}, we define
\begin{align}
  Y[S^{-1}] &= G(F)\bs\left( X_S \times G(\A^S)/K^S\right).
  \label{Y_S}
\end{align}
This generalizes the construction of the space $Y_0(1)$ used in Section 
\ref{sec:K2}, which corresponds
to the case $S = \Vi$. Similarly as for $Y_0(1)$, 
we have that $Y[S^{-1}]$ consists of at most $h_F$
copies of the quotient $\PGL_2(\O_F(S))\bs X_S$.

Although it is not stated explicitly in the work of Calegari and
Venkatesh, the result that follows is obtained by the same spectral sequence
argument used to prove \cite[Theorem 4.5.1]{CalVen}. See in particular
\cite[\S 4.5.7.5]{CalVen}, where the case $S = \Vi \cup
\left\{ \q \right\}$ is used.

\begin{theorem}[Calegari-Venkatesh]
  \label{thm:CV-S-arithm}
  Let $F$ be a number field and $S$ be a finite set of places of $F$
  containing the archimedean and of size $|S| > 1$. Let $R =
  \Z[\frac{1}{6w_F}]$.
  Then, there exists a surjective map
  \begin{align*}
    H_2(Y[S^{-1}],R) &\to K_2(\O_F(S)) \otimes R.
  \end{align*}
\end{theorem}

The case $N = 1$ of the following theorem corresponds to Theorem
\ref{thm:K2}, and in particular it supplies the proof of the latter for
the cases $d=2$ and $d=4$, which remained unproved in Section \ref{sec:K2}.

\begin{theorem}
  \label{thm:K2-S-arithm}
  Let $d, N \ge 1$ be two integers with $d$ even. There exists a constant
  $C(d,N)$ such that for any totally imaginary number field $F$ of
  degree $d$ and any set of places $S$ of $F$ of norm $N_{F/\Q}(S) \le N$,
  we have 
  \begin{align*}
    \log|K_2(\O_F(S)) \otimes R|  &\le C(d,N) |D_F|^{2}
    (\log|D_F|)^{d-1},
  \end{align*}
  where $R = \Z[\frac{1}{6w_F}]$.
\end{theorem}

\begin{proof}
  Since for $S \subset S'$ we have $K_2(\O_F(S)) \subset K_2(\O_F(S'))$, 
  by adding at most two places (of bounded norm) to $S$ we may assume
  that $|\Sf| \ge 3$, so that Theorem \ref{thm:torsion-SL-S} applies.
  Now, the quotient of $\PGL_2(\O_F(S))$ by $\PSL_2(\O_F(S))$ is 
  given by
  \begin{align*}
    \O_F(S)^\times/\left( \O_F(S)^\times \right)^2,
  \end{align*}
  whose order is invertible in $R$. We deduce that 
  \begin{align}
    |H_2(\PGL_2(\O_F(S)),R)| &\le |H_2(\PSL_2(\O_F(S)),R)|.
    \label{eq:PSL2-PGL2}
  \end{align}
  From Theorem \ref{thm:CV-S-arithm}, Theorem \ref{thm:torsion-SL-S} and
  \eqref{eq:class-no}, we have 
  \begin{align*}
    \log|K_2(\O_F(S))\otimes R| &\le  h_F \cdot
    \log|H_2(\PGL_2(\O_F(S)),R)|\\
    &\le C(d,N) \cdot |D_F|^2 (\log|D_F|)^{d-1}, 
    \label{eq:bound-to-PGL}
  \end{align*}
  for some constant $C(d,N)$.
\end{proof}

\bibliographystyle{amsplain}
\bibliography{K2}

\end{document}